\documentclass[10pt]{article}
\usepackage{amssymb,amsthm}
\usepackage[usenames]{color}
\usepackage[utf8]{inputenc}
\usepackage{amsmath,amsfonts,amscd,amssymb,amsthm,latexsym,eucal,graphics,amsbsy} 
\newtheorem{lem}{Lemma}
\newtheorem*{theorem}{Theorem}
\newtheorem{defi}{Definition}

\newcommand{\prk}{\operatorname{prk}}

\begin{document}

\begin{center}
{\bf Characterization of groups with non-simple socle}

{\bf I. B. Gorshkov}
\medskip
\footnote{The work was supported by
RFBR 18-31-00257.}

\end{center}
{\it Abstract: The spectrum of a finite group is a set of its element orders. We prove that if $m>5$ then the group $L_{2^m}(2)\times L_{2^m}(2)\times L_{2^m}(2)$ is uniquely determined by its spectrum in the class of finite groups.

\smallskip

Keywords: finite group, spectrum. \smallskip
}

\section*{Introduction}
Let $\omega(G)$ be the spectrum of group $G$. A finite group $L$ is recognizable by spectrum if every finite group $G$ with $\omega(G)=\omega(L)$ is isomorphic to $L$. W.J. Shi showed \cite{Shi} that if a finite group has a non-trivial soluble radical, then there are an infinite number of non-isomorphic finite groups with the same spectrum. Thus, the socle of a recognizable group is a direct product of nonabelian simple groups. The question on recognizability of a finite nonabelian simple groups is solved for many groups \cite{VasBig}. If the socle of $G$ is a finite simple group but $G$ is not simple, then $G$ can be recognizable. For example, in \cite{GorSym, GorGri} it is proved that symmetric groups of dimension greater than 10 are recognizable. If the socle of $G$ is not a simple group, then two recognizable groups are known, these are $J_4\times J_4$ see \cite{GorMas}, and $Sz(2^7)\times Sz(2^7)$ see \cite{Maz97}. In the present paper we prove that the direct product of three groups $L_{2^m}(2)$, where $m>5$, is recognizable.

\begin{theorem}
The groups $L_{2^m}(2)\times L_{2^m}(2)\times L_{2^m}(2)$, where $m>5$ are recognizable.
\end{theorem}

\section{Preliminary results}

In this paper, all groups are finite. Denote $G^n=\underbrace{G\times...\times G}_n$.

\begin{defi}
Let $H$ be a finite group, $n$ be a positive integer.
\begin{enumerate}
\item{$\sigma(n)=$ number of different primes dividing $n$,}
\item{$\sigma(H)=max\{\sigma(|g|)| g\in H\}$,}
\item{$\gamma(n)=max\{\sigma(|T|)| T$ is a solvable group with $\sigma(T)=n\}$,}
\item{$\pi(H)=\sigma(|H|)=$ set of prime divisors $|H|$,}
\end{enumerate}
\end{defi}

\begin{lem}\cite[Theorem 3.3]{Keller94}\label{sigmas}
$\gamma(3)=8$.
\end{lem}

\begin{lem}\cite{BuA}\label{Spectrum}
Let $G =L^{\varepsilon}_{n}(q)$, where $n\geq2$, $q$ is a power of $p$. Put $d=(n, q-\varepsilon1)$. Then $\omega(G)$ consists from the devisers of
\begin{enumerate}
\item{$\frac{q^n-(\varepsilon1)^n}{d(q-\varepsilon1)}$;}
\item{$\frac{[q^{n_1}-(\varepsilon1)^{n_1}, q^{n_2}-(\varepsilon1)^{n_2}]}{(n/(n_1,n_2),q-\varepsilon1)}$ where $n_1, n_2 >0$ are such that $n_1+n_2=n$;}
\item{$[q^{n_1}-(\varepsilon1)^{n_1}, q^{n_2}-(\varepsilon1)^{n_2}, . . . , q^{n_s}-(\varepsilon1)^{n_s}]$ where $s > 3$ and $n_1, n_2, . . . , n_s>0$ are such that $n_1 + n_2 + . . . + n_s = n$;}
\item{ $p^{k \frac{q^{n_1}-(\varepsilon1)^{n_1}}{d}}$ where $k, n_1 > 0$ are such that $p^{k-1} + 1 + {n_1}=n$;}
\item{$p^k[q^{n_1}-(\varepsilon1)n_1, q^{n_2}-(\varepsilon1)^{n_2}, . . . , q^{n_s}-(\varepsilon1)^{n_s}]$ where $s> 2$ and $k, n_1, n_2 . . . , n_s > 0$ are such that $p^{k-1} + 1 + n_1 + n_2 + . . . + n_s = n$;}
6) $p^k$, where $p^{k-1} + 1 = n, k > 0$.
\end{enumerate}
\end{lem}

If $n$ is a nonzero integer and $r$ is an
odd prime with $(r, n) = 1$, then $e(r, n)$ denotes the multiplicative order of $n$ modulo $r$.
Given an odd integer $n$, we put $e(2, n)=1$ if $n\equiv 1(mod\ 4)$, and $e(2, n) = 2$ otherwise.
Fix an integer $a$ with $|a| > 1$. A prime $r$ is said to be a primitive prime divisor of
$a^i-1$ if $e(r, a)=i$. We write $r_i(a)$ to denote some primitive prime divisor of $a^i-1$, if such a prime exists, and $R_i(a)$ to denote the set of all such divisors.
Zsigmondy \cite{Zs} proved that primitive prime divisors exist for almost all pairs $(a,i)$.

\begin{lem}(Zsigmondy). Let $a$ be an integer and $|a| > 1$. For every natural number
$i$ there exists some primitive prime devisors $r_i(a)$, except for the pairs $(a,i)\in\{(2, 1), (2, 6), (-2, 2), (-2, 3),
(3, 1), (-3, 2)\}$.
\end{lem}

Let $\Theta$ be a set of primes. We say that $\Theta$ is an independent set of a group $H$ if $\Theta\subseteq \pi(H)$ and for any distinct $p,q\in \Theta$ we have $pq\not\in\omega(H)$. Denote by $\rho(G)$ some maximal independent set of $G$ and $t(G)=|\rho(G)|$.
We call the number $t(G)$ the independence number of the group $G$.

\begin{lem}\cite[Lemma 2.2]{VasBig}\label{L8} Let $L$ be a simple classical group over a field of order $q$ and characteristic $p$. Suppose that $r$ and $s$ are distinct primes, $r, s \not\in\pi(q^2-1)$, $r\in R_i(q)$, and $s\in R_j(q)$.

(i)If $rs \in \omega(L)$, then $r's'\in\omega(L)$ for every distinct odd primes $r'\in R_i(q)$ and $s'\in R_j(q)$.

(ii)If $pr\in(L)$, then $pr'\in\omega(L)$ for every odd prime $r'\in R_i(q)$.
\end{lem}

For a classical group $L$, we put $prk(L)$ to denote its dimension if $L$ is a linear or unitary group, and its Lie rank if $L$ is a symplectic or orthogonal group. Now we introduce a new function in order to unify further arguments on classical groups. Namely, given a simple classical group $L$ over a field of order $q$ and a prime $r$ coprime to $q$, we put

$$\varphi(r,L)=\begin{cases}
e(r, \varepsilon q),& \text{ if } L=L^{\varepsilon}_n(q),\\
\eta(e(r, q)),& \text{ if } L \text{ is symplectic or orthogonal}.
\end{cases}$$

Where $\eta(a)=a$ if $a$ is odd, and $\eta(a)=a/2$ if $a$ is even.

\begin{lem}\cite[Lemma 2.4]{VasBig}\label{svyz}
Let $L$ be a simple classical group over a field of order $q$ and characteristic $p$,
and let $prk(L)=n\geq4$.

(i) If $r\in\pi(L)\setminus\{p\}$, then $\varphi(r,L)\leq n$.

(ii) If $r$ and $s$ are distinct primes from $\pi(L)\setminus\{p\}$ with $\Phi(r,L)\leq n/2$ and $\Phi(s,L)\leq n/2$,
then $rs\in\omega(L)$.

(iii) If $r$ and $s$ are distinct primes from $\pi(L)\setminus\{p\}$ with $n/2<\Phi(r,L)\leq n$ and $n/2<
\Phi(s,L)\leq n$, then $r$ and $s$ are adjacent in $GK(L)$ if and only if $e(r, q)=e(s, q)$.

(iv) If $r$ and $s$ are distinct primes from $\pi(L)\setminus\{p\}$ and $e(r, q) = e(s, q)$, then $r$ and $s$
are adjacent in $GK(L)$.

\end{lem}

\begin{lem}\cite[Lemma 3.5]{VasBig}\label{action}
Let $L$ be a simple classical group over a field of order $q$ and characteristic $p$,
$r\in\pi(L), r^s\in\omega(P)$, where $P$ is a proper parabolic subgroup of $L$, and $(r, 6p(q +1))=1$.
If $L$ acts faithfully on a vector space $V$ over the field of characteristic $t$ distinct from $p$,
then $tr^s\in\omega(V\leftthreetimes L)$.
\end{lem}

\begin{lem}\cite[Lemmas 3.3,~3.6]{VasBig}\label{fact}
Let $s$ and $p$ be distinct primes, a group $H$ be a semidirect
product of a normal $p$-subgroup $T$ and a cyclic subgroup $C=\langle g\rangle$ of order $s$,
and let $[T, g]\neq 1$. Suppose that $H$ acts faithfully on a vector space $V$ of
positive characteristic $t$ not equal to $p$.

If the minimal polynomial of $g$ on $V$ equals to $x^s-1$, then $C_V(g)$ is non-trivial.

If the minimal polynomial of $g$ on $V$ does not equal $x^s-1$, then

$($i$)$ $C_T(g)\neq 1$;

$($ii$)$ $T$ is nonabelian;

$($iii$)$ $p=2$ and $s = 2^{2^{\delta}}+1$ is a Fermat prime.
\end{lem}


\begin{lem}{\rm\cite[Lemma 3.8]{VasBig}}\label{l:adjanisotrop}
For a simple classical group  $S$ over a field of order $u$ and characteristic $v$ with
$\prk(S)=m\geq4$, put

$$j=\begin{cases}
m & \text{ if }S\simeq L_{m}(u),\\
2m-2 & \text{ if either }S\simeq O^+_{2m}(u)\text{ or }S\simeq U_{m}(u)\text{ and }m\text{ is even,}\\
2m & \text{ otherwise.}
\end{cases}$$

Then $(r_j(u),|P|)=1$ for every proper parabolic
subgroup $P$ of~$S$. If $i\neq j$ and a primitive prime divisor $r_i(u)$ lies in $\pi(L)$ then
there is a proper parabolic subgroup $P$ of $S$ such that $r_i(u)$ lies in~$\omega(P)$. In
particular, if two distinct primes $r,s\in\pi(S)$ do not divide the order of any proper parabolic
subgroup of $S$ then $r$ and $s$ are adjacent in $\Gamma(S)$.
\end{lem}

\begin{lem}\cite[Lemma 10]{ZavMaz}\label{Zav}
Let $N$ be a normal elementary abelian $p$-subgroup of $G$, $K=G/N$, and $G_1=NK$ be
the natural semidirect product. Then $\omega(G_1)\subseteq\omega(G)$.
\end{lem}

\begin{lem}\cite[Lemma 5]{GrVas}\label{Frobenius}
Let $L$ be a finite simple group $L_n(q)$, $d=(q-1,n)$.
If there exists a primitive prime divisor $r$ of $q^n-1$ then $L$ includes a Frobenius subgroup with
kernel of order $r$ and cyclic complement of order $n$.

\end{lem}

Let $L=L_{2^m}(2)$, where $m>5$, $R=L^3$. Suppose that $G$ is a group such that $\omega(G)=\omega(R)$.
Set $r_i=r_i(2)$ and
$\Omega=\{r_{2^{m-1}},r_{2^{m-1}+1},...,r_{2^m}\}$.

\begin{lem}
$t(L)=2^{m-1}$.
\end{lem}
\begin{proof}
The assertion of the lemma follows from Lemma \ref{Spectrum}.

\end{proof}

\begin{lem}\label{independent}
$\Omega$ is a maximal independent set of $L$.
\end{lem}
\begin{proof}
The assertion of the lemma follows from Lemma \ref{Spectrum}.

\end{proof}
\begin{lem}\label{2spek}
$2^m\in\omega(L)$ and $2^{m+1}\not\in\omega(L)$.
\end{lem}
\begin{proof}
The assertion of the lemma follows from Lemma \ref{Spectrum}.
\end{proof}

\begin{lem}\label{spectr}
If $g\in G$ and $\pi(g)\subseteq\Omega$, then $|\pi(g)|\leq 3$.
\end{lem}
\begin{proof}
The proof is trivial.
\end{proof}

\section{Proof Theorem}

Let $L=L_{n}(2)$, where $n=2^m$ and $m>5$, $R=L^3$. Suppose that $G$ is a group such that $\omega(G)=\omega(R)$.

Denote $1<G_1<G_2<... <G_r=G$ is a principal series of $G$, $S_i=G_i/G_{i-1}$.

\begin{lem}\label{qk}
$\omega(G)$ does not contain $q^{n/2-1}-1$, where $q$ is odd.
\end{lem}
\begin{proof}
We have $q^{n/2-1}-1 =(q-1)(q+1)(q^2 + 1). . .(q^{2m-2}+ 1)$. If follows that $(q^{n/2-1}-1)_2>n$; a contradiction with Lemma \ref{2spek}.
\end{proof}

\begin{lem}\label{fak}
Let $S$ be a composition factor of $G$. If $S$ is a group of Lie type over a field of odd characteristic, then Lie rank of $S$ does not exceed $n/2+3$.
\end{lem}
\begin{proof}
Assume that $G$ includes a composition factor $S\simeq\Lambda_k(q)$, where $\Lambda_k(q)$ is a group of Lie type, $k>n/2+2$.

From the description of maximal tories (see \cite{car1, car2}) it follows that $q^{r}-1\in\omega(S)$ for each even $r\leq k-1$. The assertion of lemma follows from Lemma \ref{qk}.







\end{proof}

\begin{lem}\label{ryd}
Let $\Lambda=\{i_1, i_2, ..., i_k\}$ be a set of numbers such that there exists $p_j\in\pi(S_j)\cap(\Omega\setminus\{p_i| i\in\Lambda \setminus\{j\}\})$, where $j\in \Lambda$. Then $|\Lambda|\leq 8$.
\end{lem}
\begin{proof}
The proof is similar to that \cite[Lemma 1.1]{Vas05}. Let $\Sigma=\{p_j|j\in \Lambda\}|$, where $p_j$ as from the hypothesis of the lemma.
Assume that $|\Sigma|>8$. Using Frattini's argument, it is easy to obtain that $G$ includes a subgroup $H$ such that $\pi(H)= \Sigma$. Since $2\not\in\pi(H)$, the group $H$ is solvable. Lemma \ref{sigmas} implies that $H$ contains an element $g$ whose order is divisible by $4$ different numbers from the set $\Sigma$; a contradiction.

\end{proof}

Let $\Lambda=\{i_1, i_2, ..., i_k\}$ be a maximal set of indices such that there exits a set $\Sigma\leq\Omega$ and $p_j\in\pi(S_{i_j})$, where $p_j\in \Sigma$, $i_j\in \Lambda$. It follows from Lemma \ref{ryd}  that $|\Lambda|\leq 8$.


\begin{lem}\label{Simpleryd}
The set $\{S_i|i\in\Lambda\}$ contains at most $3$ nonsolvable groups.
\end{lem}
\begin{proof}
Let $S_a, S_b, S_c, S_d$ be nonsolvable groups such that $a<b<c<d$ and $a,b,c,d\in \Lambda$. Let us show that in this case there is an element $g\in G$ such that $\pi(g)= p_ap_bp_cp_d$, where $ p_a, p_b, p_c, p_d $ are distinct numbers from $\Omega$. Since $S_c$ is the principal factors of $G$, then $S_c \simeq C_1 \times C_2 \times ... \times C_r$, where $C_i$ are isomorphic simple groups. The automorphism group of any simple group is solvable. Therefore, any simple subgroup of $G /G_{c}$ acts trivially on the diagonal subgroup of $S_c$. Hence $\widetilde {G} = G / G_ {c-1} $ includes a subgroup $ H \simeq C_1 \times D_1$, where $D_1$ is the minimal preimage of a minimal normal subgroup of the group $S_d$. In particular, $H$ contains an element of order $p_cp_d$. Similarly, we see that $G$ contains an element of order $p_ap_bp_cp_d$; a contradiction with Lemma \ref{spectr}.

\end{proof}

Let $\overline{\Lambda}\leq\Lambda$ be such that $S_i$ is nonsolvable, where $i\in \overline{\Lambda}$. It follows from Lemma \ref{Simpleryd} that $|\overline{\Lambda}|\leq 3$.

\begin{lem}\label{four}
There exists $j\in\overline{\Lambda}$ such that $|(\pi(S_j)\cap \Omega)|\geq 9$.
\end{lem}
\begin{proof}
We have $|\Omega|=n/2\geq 32$. Since the order of each solvable principal factor is divisible by at most one number from $|\Omega|$, we have at least $25$ numbers from $\Omega$ divide orders of nonsolvable principal factors of $G$. It follows from Lemma \ref{Simpleryd} that $G$ contains at most $3$ nonsolvable principal factors with this satisfies this conditions. Therefore, order of at least one of them is divisible by $9$ numbers from $\Omega$ 

\end{proof}

\begin{lem}\label{good}

There exists $H\lhd G$ such that $X=Soc(G/H)=A_1\times A_2\times A_3$, where $A_i$ is a nonabelian simple group and $|\pi(A_i)\cap\Omega|>3$, for $1\leq i\leq 3$.

\end{lem}
\begin{proof}
By Lemma \ref{four} there exists a principal factor $S$ of $G$ such that $|\pi(S)\cap\Omega|>3$. Let $j$ be a minimal number such that $|\pi(S_j)\cap\Omega|>3$. 

Let $X$ be a socle of $\widetilde{G}=G/G_{j-1}$. Put $\Delta= \Omega\cap(\pi(X)\setminus(\pi(G_{j-1})\cup\widetilde{G}/X))$.
We can assume that $X$ is a direct product of simple groups $A_1,..., A_k$ and the order of each of them is divisible by at least $4$ numbers from $\Omega$. Therefore, $k\leq 3$. We have $\widetilde{G}\leq Aut(X)$. Since the group of outer automorphisms of a simple group is solvable, it follows that $\widetilde{G}/X$ is solvable.

For each $p, q, l \in\Delta$ the group $X$ contains an element of order $pql$.
Assume that $G_{j-1}$ includes a nonsolvable composition factor whose order is divisible by $t\in\Omega$. Using Lemma \ref{Simpleryd} it is easy to see that in this case $G$ contain an element of order $pqrt$; a contradiction.
We have $|\Delta|\geq|\Omega|-8+k$. 

Assume that $X$ is a simple group.
Assume that $X\simeq Alt_t$. In this case easy to see that $|\Delta|\leq \lambda(t/3,t/4)<\frac{t}{12}$, where $\lambda(a, b)$ is a number of primes in the interval $[a,b]$. Therefore $t>2n+2$ and $2n\in\omega(X)$; a contradiction.

Assume that $X$ is isomorphic one of the sporadic simple groups. Since $\Delta$ contains a number greater than $100$, a contradiction.

Using the information about orders of the maximal tori of exceptional groups of Lie type \cite{car2}, \cite{Der}, we get that $X$ does not isomorphic to an exceptional group of Lie type.


Assume that $X$ is isomorphic to $\Lambda_t(q)$, a classical simple group of Lie type and $prk(\Lambda_t(q))=t$ over a field of order $q$. Assume that there exists $i$ such that $p,r\in R_i(q)$, for some $p,r\in \Delta$. Then if $ar\in \omega(X)$, then from Lemma \ref{L8} follows that $arp\in \omega(X)$, where $a\in\pi(X)\setminus\{p,r\}$. Let $v,w\in\Delta\setminus \{p,r\}$. Therefore $prvw\in\omega(X)$; a contradiction. Therefore if $p\in R_i(q)$ and $r\in R_j(q)$, then $i\neq j$. From the fact that for each $p,r\in \Delta$ we have $pr\in\omega(X)$ and Lemma \ref{svyz} it follows that $\varphi(s,X)<t/2$ for each $s\in \Delta$ except one. It follows from Lemma \ref{four} that $\Delta\geq n/2-9$. Therefore $t/2\geq n/2-9$. Hence $t\geq n-18>n/2+3$. Assume that $q$ is odd. It follows from Lemma \ref{fak} that $t\leq n/2+2$; a contradiction. 

Assume that $q=2^l$. In this case, it is easy to show that $R_i(2)\subseteq \pi(X)$, where $i>n$; a contradiction.


Assume that $X=A_1\times A_2$. We have that $|\pi(A_1)\cap \Delta|>3$ and $|\pi(A_2)\cap \Delta|>3$. Suppose there are numbers $p, r, t, s \in\Delta$ such that $pr\in\omega(A_1)$ and $ts\in\omega(A_2)$. Then $prts\in\omega(X)$; a contradiction. Therefore, $\omega(A_1)$ includes $pr$ for each $p,r\in \Delta\cap\pi(A_1)$ and $|\pi(A_1)\cap\Delta|\geq|\Delta|-1$. Similarly, as above, we get a contradiction.


\end{proof}

We fix the following notation $X$, $H$ and $A_1, A_2, A_3$ are such as in Lemma \ref{good}, $\widetilde{G}=G/H$.
Let $\Delta=\Omega\cap(\pi(X)\setminus(\pi(H)\cup\widetilde{G}/X))$.


\begin{lem}\label{coco}
The following statements are true:
\begin{enumerate}
\item For each distinct $p,q\in\Delta$ we have $pq\not\in\omega(A_i)$, where $1\leq i\leq 3$.

\item $\Delta\subset\pi(A_i)$ for each $1\leq i\leq 3$.
\end{enumerate}
\end{lem}
\begin{proof}
Assume that there exists distinct $p,q\in \Delta$ such that $pq\in\omega(A_1)$. In this case $\omega(X)$ includes $pqlt$, where $l\in\pi(A_2)\cap(\Delta\setminus\{p,q\})$ and $t\in\pi(A_3)\cap(\Delta\setminus\{p,q,l\})$; a contradiction. Therefore, if $p,q\in\pi(A_1)\cap\Delta$ then $pq\not\in \omega(A_1)$.

Since $pqr\in\omega(X)$, for a distinct numbers $p,q,r\in \Delta$, we have that $\Delta\in \pi(A_i)$, where $1\leq i\leq 3$.


\end{proof}

\begin{lem}\label{Class}
The $A_i$ is a classical group of Lie type over a field of even characteristic, where $1\leq i\leq 3$.
\end{lem}
\begin{proof}
From the fact that $\Delta\in\pi(A_i)$ and for each distinct $p,q\in\Delta$ the spectrum of $A_i$ does not contain $pq$ see Lemma \ref{coco}, it follows that $A_i$ does not isomorphic one of sporadic groups or exceptional group of Lie type.

Suppose that $A_i$ is isomorphic to an alternating group of degree $l$. Then the interval $[l/2, l]$ contains at least $n/2-7$ primes. From the description of the distribution of primes it follows that $l>2n+2$. Therefore, $A_i$ contains an element of order $2n$; a contradiction.


Suppose that $A_i$ is isomorphic to a simple classical group of Lie type $\Lambda_{n'}(q)$ over a field of odd order $q$. To obtain a contradiction with Lemma \ref{qk}, it is sufficient to show that $\omega(\Lambda_{n'}(q))$ contains an element of order $q^{n/2}-1 $. The estimate for the number $n'$ will be obtained from the fact that the $t(A_i)$ is greater than or equal to $n/2-7$. From the description of the spectra of finite simple groups \cite{VV}, we obtain the required statement.

\end{proof}

\begin{lem}\label{suma}
$\Omega=\Delta$.
\end{lem}
\begin{proof}
Assume that there exists $h\in\pi(H)\cap\Omega$. Let $W<H$ be a normal subgroup of $G$ of maximal order such that $|H/W|$ is a multiple of $h$. Put $\overline{H}=H/W$, $Y\unlhd\overline{H}$ is a minimal normal in $G$ subgroup, $\overline{G}=G/W$. We can assume that $|\overline{G}/Y|$ is not divisible by numbers from $\Omega$ and $C_{\overline{H}}(Y)=Y$. Since $|\Delta|> 27$, by Lemma \ref{l:adjanisotrop} there exists pairwise distinct numbers $p_1, p_2, p_3 \in \Delta$ such that $p_i\in \pi P_i $, where $P_i <A_i $ is a parabolic subgroup. Let $a_i\in P_i$ be an element of order $p_i$, where $1\leq i\leq 3$. We have $\overline{G}$ contains the subgroup $V.(P_1\times P_2 \times P_3)$. Since $|V|$ is not divisible by numbers from $\Omega$, using Lemma \ref{fact}, we obtain that $\overline{G}$ contains an element of order $hp_1p_2p_3$; a contradiction.
Thus, $|H|$ is not divisible by numbers from $\Omega$.

Suppose that $|\widetilde{G}/X| $ is divisible by some number $h$ from $\Omega$. Then there exists $1\leq i\leq3$ such that $h\in \pi (Out (A_i)$. Since $h>32$ and the characteristic of the task field of the group $A_i$ is equal $2$, then $A_i$ contains the field automorphism of order $h$. Let $k=2^t$ be the order of the field over which the group $A_i$ is given. Then $h$ divides $t$, in particular, $t>n/2$. In this case it is easy to show that $A_i$ contains a primitive prime divisor of the number $2^{r}-1$, where $r>n$; a contradiction.
Thus, $|H|$ and $ |\widetilde{G}/X| $ are not divisible by numbers from $\Omega$ and therefore $\Delta=\Omega$.

\end{proof}

\begin{lem}
$A_i\simeq L_{2^m}(2)$, where $1\leq i\leq3$.
\end{lem}
\begin{proof}
It follows from Lemma \ref{Class} that $A_i$ is a classical group over a field of even characteristic.

Assume that $A_i\simeq L_t(2^k)$. Since $r_{tk}\in \pi(L_t(2^k))$ we have $t+k\leq 2^m$. It follows from the fact $t(A_i)\geq2^{m-1}$ that $t=2^m$. Therefore $k=1$.

Assume that $A_i\simeq U_t(2^k)$. It follows from the fact $t(A_i)=2^{m-1}$ that $t\geq2^m$ or $t=2^m-1$. In this cases $\pi(A_i)$ contains $r_{2(2^{m}-1)}(2^k)$; a contradiction.

If $A_i$ is a symplectic or an orthogonal group, then from description of maximal tori follows that $\pi(A_1)$ contains $r_{2t(A_i)+1}(2)$; a contradiction.

\end{proof}

\begin{lem}\label{volna}
$\widetilde{G}=X$.
\end{lem}
\begin{proof}
Assume that there exists an element $g\in N_{\widetilde{G}}(A_1)$ which acts on $A_1$ as an outer automorphism. Since $Aut(A_1)\simeq A_1\leftthreetimes \lambda$ where $\lambda$ is the graph automorphism of order $2$, we have $g$ acts on $A_i$ as the graph automorphism. It follows from \cite{Wall} that in this case $\omega(\widetilde{G})$ contains the number $2^{m+1}$; a contradiction.

Therefore, if $g\in\widetilde{G}\setminus X$ then without loss of generality we can assume that $A_1^g\neq A_1$. In this case $\omega(\widetilde{G})\setminus \omega(R)$ contains a number $2^l$ or $3^t$; a contradiction.

\end{proof}

\begin{lem}
$X=G$.
\end{lem}
\begin{proof}
It follows from Lemma \ref{volna} that $G=H.X$.
Assume that $H$ is a nontrivial group. Without loss generation we can think that $H$ is an abelian $p$-group. It follows from Lemma \ref{suma} that $p\not\in\Omega$.

Suppose that $H.X$ is a central extension of $H$ by $X$. Lemma \ref{Zav} implies that it suffices to prove that $\omega(H\times X)\neq\omega(R)$. To do this, we need to find the maximal by divisibility number $\gamma \in \omega(X)$ that is not divisible by $p$. Assume $p=2$. There is a decomposition of the number $2^m-1$ into the sum of odd numbers $k_1+k_2+...+k_l$ such that $k_i$ does not divide $k_j$ for $ i, j \in \{1..l \}, i \neq j$. Let $g_1\in A_1$ be an element of order $(2^{k_1}-1)(2^{k_2}-1)...(2^{k_l}-1)$, $g_2\in A_2$ an element of order $2^{n}-1$, $g_3\in A_3$ an element of order $2^{n-1}-1$. Obviously, $|g_1g_2g_3|$ is the maximal by divisibility number in $\omega(R)$ and not divisible by $2$. Hence $p\neq 2$. Let $e(p, 2)=t$. There are sets $\Upsilon$, $\Phi$, $\Psi$ of positive integers such that $a_i$ does not divide $a_j$, where $a_i, a_j \in \Upsilon \cup \Phi \cup \Psi$ are distinct numbers and $(\Upsilon \cup \Phi \cup \Psi)\cap \{t,1\}=\varnothing$, $\sum \Upsilon=\sum \Phi=2^m, \sum \Psi = 2^m-1 $. The number $\prod_{i \in \Upsilon \cup \Phi \cup \Psi}2^i-1$ is maximum by divisibility in $\omega(R)$ and is not divisible by $p$.

Thus, we can assume that $A_1$ acts faithfully and irreducibly on $H$. Suppose $p=2$. Lemma \ref{Frobenius} implies that $A_1$ contains a Frobenius group with kernel of order $2^{2^m}-1$ and complement of order $2^m$.
It follows from Lemma \ref{action} that $\omega(G)$ contains the number $2^{m+1}$; a contradiction with Lemma \ref{2spek}. Therefore $p\neq 2$. Let $p^l\in\omega(A_1)$ be such $p^{l+1}\not\in \omega(A_1)$. Since $p\not\in R_{2^m}(2)$ it is easy to show that there exists $t<2^m$ such that $p^l| 2^t-1$. The group $A_1$ include a parabolic subgroup $P$ such that $p^l\in\omega(P)$. In particular $A_1$ includes a group $NM$ where $N$ is a normal $2$-subgroup of $NM$ and $M$ is a cyclic group of order $p^l$. It follows from Lemma \ref{action} that $p^{l+1}\in\omega(G)$; a contradiction.
\end{proof}




\end{document}